\documentclass[fleqn,11pt]{article}
\usepackage{amsmath,amssymb,amsfonts,amsthm, array, booktabs}
\usepackage{color}

\newtheorem{Theorem}{Theorem}

\newtheorem{Example}[Theorem]{Example}
\newtheorem{Proposition}[Theorem]{Proposition}
\newtheorem{Definition}[Theorem]{Definition}
\newtheorem{Corollary}[Theorem]{Corollary}
\newtheorem{Lemma}[Theorem]{Lemma}
\newtheorem{Remark}[Theorem]{Remark}

\begin{document}

\title{A Quaternionic Hansen--Pedersen Inequality and its Application
to $S_r$-Transforms}
\author{ M. Fashandi\footnote{{\small{E-mail:}  fashandi@um.ac.ir (M. Fashandi)}} \\
\footnotesize{  Faculty of Mathematical Sciences, Ferdowsi University of Mashhad,}\\
  \footnotesize{P.O. Box 1159,  Mashhad, 91775\ Iran }
  }

\date{}
\maketitle

\begin{abstract}We establish a quaternionic Hansen--Pedersen inequality for real
operator-convex functions on right quaternionic Hilbert spaces. The
proof is based on the continuous real functional calculus for
selfadjoint quaternionic operators, a direct-sum functional-calculus
identity, and a block-unitary construction. As an application, let
\(T = U|T|\) be an injective \(p\)-hyponormal operator. We prove that,
for \(0 < r \leq 1/2\), the transform \(S_{r}(T) = U|T|^{r}U\) is
\(2p\)-hyponormal when \(0 < p \leq 1/2\) and hyponormal when
\(1/2 < p \leq 1\). We also show that \(S_{r}(T)\) preserves
log-hyponormality for \(r > 0\). These results illustrate the use of
quaternionic Jensen-type operator inequalities in the study of
transforms associated with the polar decomposition.
\end{abstract}



\maketitle

%%%%%%%%%%%%%%%%%%%%%%%%%%%%%%%%
\section{Introduction}

The study of operator inequalities on Hilbert spaces has a long and rich history, with fundamental contributions from L\"owner, Heinz, Furuta, Hansen, and many others. These inequalities have proven to be indispensable tools in various areas of operator theory, including the study of hyponormal, $p$-hyponormal, and paranormal operators, as well as the analysis of transforms related to the polar decomposition.

In the complex setting, Hansen--Pedersen inequality in \cite{Hansen2003} is a cornerstone result in operator convexity theory. It states that if $X$ is a self-adjoint operator, $f$ is a real operator convex function, and $V$ is an isometry, then
\[f\left( V^{*}XV \right) \leq V^{*}f(X)V.\]
This inequality, originally due to Hansen \cite{Hansen1990}, has found numerous applications, including proofs of the Furuta inequality, the Choi-Davis-Jensen inequality, and various norm inequalities for $p$-hyponormal operators.
In recent years, there has been growing interest in extending classical operator theory to the quaternionic setting, motivated by applications in mathematical physics, quantum mechanics, and Clifford analysis. The development of the S-functional calculus in \cite{Col-book} has provided a robust framework for defining functions of quaternionic operators. Building upon this foundation, the quaternionic versions of the L\"owner-Heinz and H\"older-McCarthy inequalities were established in \cite{Adv.}, and more recently, the quaternionic Furuta inequality was proved in \cite{JANA}, opening the way for the study of $p$-hyponormal operators and their Aluthge transforms in the quaternionic setting. Very recently, $q$-numerical radius for quaternionic operators have been investigated in \cite{Moulaharabbi2025}, while the equivalence between the S-spectrum and the right spectrum of quaternionic operators was  established in  \cite{Carvalho2025}. These developments further motivate the study of operator inequalities and transforms in the quaternionic setting.
To the best of our knowledge, a quaternionic analogue of the Hansen--Pedersen operator inequality for bounded right-linear operators on quaternionic Hilbert spaces has not previously been established. The  main contribution of this paper is to fill this gap. In Section 3, we establish the quaternionic Hansen--Pedersen inequality (Theorem \ref{Hansen}). The quaternionic extension is nontrivial, the proof must be formulated using the quaternionic
  S-functional calculus, and the block argument must
  be checked in the category of bounded right-linear operators.

 The proof follows the elegant operator matrix method of Pe\v{c}ari\'{c} et al. \cite{Mond}, adapted to the quaternionic setting using the S-functional calculus and the quaternionic functional calculus on direct sums (Proposition \ref{directsum}). This result is new and provides a powerful tool for deriving operator inequalities in quaternionic Hilbert spaces.

The rest of the paper is organized as follows. Section 2  contains the necessary preliminaries on quaternionic Hilbert spaces, the S-functional calculus, polar decomposition, and known quaternionic operator inequalities. In Section 3, we establish  the functional calculus on direct sums and the quaternionic Hansen--Pedersen inequality. In Section 4, we apply Hansen--Pedersen inequality to study the transform $S_r(T)$ and prove the quaternionic version of Theorem 3.1 of \cite{Srt}. Finally, we prove that \(S_{r}(T)\) preserves
log-hyponormality for \(r > 0\).
%%%%%%%%%%%%%%%%%%%%%%%%%%%%%%%%%%%%%%%%%%%%%%%%%%
\section{Preliminaries and Auxiliary Results}
In this paper, we will denote the skew field of quaternions by $\mathbb{H}$. The elements of $\mathbb{H}$ can be expressed in the form \( q = x_0 + x_1i + x_2j + x_3k \), where \( x_0, x_1, x_2, \) and \( x_3 \) are real numbers. In this context, \( i, j, \) and \( k \) are referred to as imaginary units, which follow specific multiplication rules as follows:
\begin{equation*}
i^2=j^2=k^2=-1,\;\; ij=-ji=k,\;\; jk=-kj=i,\;\; {\rm\mbox and}\;\; ki=-ik=j.
\end{equation*}
Also, $\mathsf{H}$ denotes a right quaternionic Hilbert space, which is a linear vector space over $\mathbb{H}$ under right scalar multiplication and an inner product
$\langle.,.\rangle:\mathsf{H}\times \mathsf{H}\longrightarrow \mathbb{H}$
with the following properties:
\begin{enumerate}
\item[(i)] $\overline{\langle u, v\rangle}=\langle v, u\rangle$,
\item[(ii)] $\langle u, u\rangle>0$ unless $u=0$,
\item[(iii)] $\langle u,vp+wq\rangle=\langle u, v\rangle p+\langle u, w\rangle q$,
\end{enumerate}
for every $u, v, w\in \mathsf{H}$ and $p, q\in \mathbb{H}$. The quaternionic norm  $\Vert u \Vert= \sqrt{\langle u, u\rangle}$
is assumed to produce a complete metric space (see  Proposition 2.2 in  \cite{Ghi1}).
\noindent The notation $\mathfrak{B}(\mathsf{H})$ denotes the set of all bounded right linear quaternionic operators $T$ on $\mathsf{H}$ that is
 $$T(up+v)=(Tu)p+Tv,$$
for all $u, v \in \mathsf{H}$ and $p\in \mathbb{H}$. By  Proposition 2.11 in \cite{Ghi1}, $\mathfrak{B}(\mathsf{H})$  is a complete normed space with the norm defined by
\begin{equation*}\label{norm1}
\Vert T\Vert= \sup \left\{\dfrac{\Vert Tu\Vert}{\Vert u\Vert}, 0\neq u\in\mathsf{H}\right\}.
\end{equation*}
\noindent From now on, $\mathsf{H}$ will stand for a right quaternionic Hilbert space, and by an operator $T\in \mathfrak{B}(\mathsf{H})$, we mean a bounded right linear quaternionic operator.
\noindent For every  $T\in \mathfrak{B}(\mathsf{H})$, there exists  a unique operator $T^*\in \mathfrak{B}(\mathsf{H})$, known as the adjoint of $T$,  such that, for all $u, v \in \mathsf{H}$, the identity $\langle Tu, v\rangle=\langle u, T^*v\rangle$ holds. Several properties of the adjoint operator are established in Theorem 2.15 and Remark 2.16 of \cite{Ghi1}, including $\Vert T\Vert=\Vert T^*\Vert$. It is important to note  that, the adjoint operation is not an involution on $\mathfrak{B}(\mathsf{H})$, specifically, the relation $(qT)^*=\overline{\textbf{q}}T^*$  is valid  only when  $q\in \mathbb{R}$ (see \cite{Vis}). Table  \ref{tab:operator-classes}, summarizes the definitions of various classes of quaternionic operators that have been studied in the literature. Although the present paper focuses primarily on $p$-hyponormal and related operators, the table is included as a concise literature review to situate our work within the broader landscape of quaternionic operator theory. Some of these classes are not directly employed in the sequel; however, their inclusion provides context and may serve as a foundation for future investigations.

\begin{table}[htbp]
\centering
\caption{Summary of Operator Classes}
\label{tab:operator-classes}
\begin{tabular}{>{\raggedright}p{3.8cm} >{\raggedright\arraybackslash}p{7cm} >{\raggedright\arraybackslash}p{2.2cm} >{\raggedright\arraybackslash}p{2.8cm}}
\toprule
\textbf{Class} & \textbf{Definition} & \textbf{Notation} & \textbf{Origin (Complex Setting)} \\
\midrule
Selfadjoint & \( T=T^*\) &--& Classical \\
Positive & \( \langle Tx, x\rangle \geq 0, ~ \forall x\in \mathsf{H} \) &-- & Classical \\
Unitary & \(T^*T = TT^*= I\) & -- & Classical \\
Normal & \(T^*T = TT^*\) & -- & Classical \\
Hyponormal & \(T^*T \geq TT^*\) & \(\mathcal{H}\) & \cite{Halmos1950} \\
Semi-Hyponormal & \((T^*T)^{\frac{1}{2}}\geq (TT^*)^{\frac{1}{2}}\) & \(\mathcal{SH}\) & \cite{Xia1983}\\
$p$-Hyponormal & \((T^*T)^p\geq (TT^*)^p\), \(0 < p \le 1\) & \(\mathcal{H}_p\) & \cite{Alu}\\
Log-Hyponormal & \(\log |T| \ge \log |T^*|\) ( \(T\) invertible) & \(\mathcal{LH}\) & \cite{Ando1987}\\
Class A & \(|T^2| \ge |T|^2\) & \(\mathcal{A}\) & \cite{FurutaItoYamazaki1998} \\
Class A$(p,p)$ & \((T^*|T|^{2p}T)^{1/(p+1)} \ge |T|^2,\; p>0\) & \(\mathcal{A}(p,p)\) & \cite{FurutaItoYamazaki1998}\\
Class A$(s,t)$ &
\(|T^*|^{2t}\le
\left(|T^*|^t|T|^{2s}|T^*|^t\right)^{\frac{t}{s+t}},\;
0<s,t\le1\)
& \(\mathcal{A}(s,t)\) & \cite{Tanahashi2004}\\
Paranormal &
\(\|Tx\|^2\le \|T^2x\|\,\|x\|\)
& \(\mathcal{P}\) & \cite{Istratescu1960} \\
$p$-Paranormal &
\(\||T|^pU|T|^px\|\,\|x\|
\ge
\||T|^px\|^2,\; p>0\)
& \(\mathcal{P}_p\) & \cite{Fujii2000}\\
$(p,r)$-Paranormal &
\(\||T|^{p}U|T|^{r}x\|^2
\ge
\||T|^{p}x\|\,\||T|^{r}x\|,
\)
\(p,r>0\)
& \(\mathcal{P}_{p,r}\) & \cite{YanagidaYamazaki2000} \\
Absolute $(p,r)$-Paranormal &
\(
\||T|^p|T^*|^r x\|^r \|x\|^p
\ge
\||T^*|^r x\|^{p+r},
\)
\(p,r>0\)
& \(\mathcal{AP}(p,r)\) & \cite{YanagidaYamazaki2000} \\
Absolute $k$-Paranormal &
\(\||T|^kTx\|\,\|x\|^k
\ge
\|Tx\|^{k+1},\; k>0\)
& \(\mathcal{AP}_k\) &\cite{Fujii2000}\\
$w$-Hyponormal &
\(|\widetilde{T}|^2\ge |T|^2\ge |(\widetilde{T})^*|^2\)
& \(\mathcal{W}\) & \cite{AluthgeWang2000}\\
$\omega$-Hyponormal &
\(|\widetilde{T}|\ge |T|\ge |(\widetilde{T})^*|\)
& \(\mathcal{W}_{\omega}\) & \cite{AluthgeWang2000}\\
Dominant &
\(|T|\ge U|T^*|U^*\)
& \(\mathcal{D}\) & \cite{Stampfli1970}\\
Class \(F(p,r,q)\) &
\((|T^*|^r|T|^{2p}|T^*|^r)^{1/q}
\ge
|T^*|^{2r},\;
p,r,q>0\)
& \(\mathcal{F}(p,r,q)\) & \cite{Fujii2000}\\
\bottomrule
\end{tabular}
\end{table}
\noindent Let \(T \in B\left( \mathsf{H} \right)\). Its absolute value is
\(|T| = \left( T^{*}T \right)^{1/2}\). The polar decomposition is
\(T = U|T|\), where \(U\) is the unique partial isometry whose initial
space is \(\ker\left( |T| \right)^{\bot}\) and whose final space is
\(\overline{RanT}\).
  According to Theorem 2.20 in \cite{Ghi1},  polar decomposition of $T$ is unique, $\ker (\vert T\vert) =\ker (T)$ and $\Vert Uu\Vert= \Vert u\Vert$ for every $u\in \ker (\vert T\vert)^{\perp}$.

\noindent  For  two operators $S$ and $T \in \mathfrak{B}(\mathsf{H})$, we write $S \leq T$ if $T-S$ is a positive  operator, i.e. $\langle Sx, x\rangle \leq \langle Tx, x\rangle$ for every $x\in \mathsf{H}$. In particular,  for two real numbers $m<M$, by $mI\leq T\leq MI$,  we mean
$
m\langle x, x\rangle\leq \langle Tx, x\rangle\leq M\langle x, x\rangle
$
for all $x\in \mathsf{H}$.
\begin{Definition}\label{spectrum}\cite{Ghi1}
Let $\mathsf{H}$ be a right quaternionic Hilbert space and $T\in \mathfrak{B}(\mathsf{H})$ be a right linear quaternionic operator. For $q\in \mathbb{H}$, the associated operator $\Delta_{q}(T)$ is defined by:
$$
\Delta_q(T):=T^2- T(q+\bar{q})+I \vert q\vert^2.
$$
The spherical resolvent set of $T\in \mathfrak{B}(\mathsf{H})$ is the set
$$
\rho_{S}(T):=\{q\in\mathbb{H}:\,\,\, \Delta_q(T)^{-1}\in \mathfrak{B}(\mathsf{H})\}.
$$
\end{Definition}
\noindent The \textit{spherical spectrum} $\sigma_{S}(T)$ of $T$, or briefly the S-spectrum, is defined as the complement of $\rho_{S}(T)$ in $\mathbb{H}$. A partition for $\sigma_{S}(T)$ was introduced in \cite{Ghi1}, as follows:
\begin{itemize}
\item[(i)] The \textit{spherical point spectrum} of $T$:
$$
\sigma_{pS}(T)=\{q\in\mathbb{H}; \ker(\Delta_{q}(T))\neq\{0\}\}.
$$
\item[(ii)] The\textit{ spherical residual spectrum} of $T$:
$$\sigma_{rS}(T)=\{q\in\mathbb{H}; \ker(\Delta_{q}(T))=\{0\}, \overline{Ran (\Delta_{q}(T))}\neq\mathsf{H}\}.$$
\item[(iii)] The spherical continuous spectrum of $T$:
$$\sigma_{cS}(T)=\{q\in\mathbb{H}; \ker(\Delta_{q}(T))=\{0\}, \overline{Ran (\Delta_{q}(T))}=\mathsf{H},
\Delta_{q}(T)^{-1}\notin \mathfrak{B}(\mathsf{H})\}.$$
\end{itemize}
The \textit{spherical spectral radius} of $T$, denoted by $r_{S}(T)$,  is defined by:
\begin{equation*}\label{radius}
r_{S}(T)=\sup\{\vert q\vert \in \mathbb{R}^{+}; q\in \sigma_{S}(T)\}.
\end{equation*}
An \textit{eigenvector} of $T$ with \textit{right} \textit{eigenvalue} $q$ is an element $u\in \mathsf{H}-\{0\}$, for which $Tu=uq$.
According to Proposition 4.5 in \cite{Ghi1}, $q$ is a right eigenvalue of $T$ if and only if it belongs to the spherical point spectrum $\sigma_{pS}(T)$.

\noindent Next, we recall Theorem 5.5 from \cite{Ghi1} to establish the continuous real functional calculus for selfadjoint quaternionic operators.
\begin{Theorem}\label{remark}\cite{Ghi1}
Let $\mathsf{H}$ be a right quaternionic Hilbert space and  $T\in\mathfrak{B}(\mathsf{H})$ be selfadjoint. Then, there exists a unique continuous homomorphism
\begin{eqnarray*}
\Phi_T: \mathcal{C}(\sigma_S(T), \mathbb{R})\ni f\mapsto f(T)\in \mathfrak{B}(\mathsf{H}),
\end{eqnarray*}
of real Banach unital algebras such that:
\begin{itemize}
\item[(i)] the operator $f(T)$ is selfadjoint for every $f\in \mathcal{C}(\sigma_S(T), \mathbb{R})$;
\item[(ii)] $\Phi_T$ is positive; that is, $f(T)\geq 0$ if $f \in  \mathcal{C}(\sigma_S(T), \mathbb{R})$ and $f(t)\geq 0$ for every $t\in \sigma_S(T)$.
\end{itemize}
By $\mathcal{C}(\sigma_S(T), \mathbb{R})$ we mean the commutative real Banach unital algebra of continuous real-valued functions defined on $\sigma_S(T)$.
\end{Theorem}
\noindent Theorem \ref{Bounds} from \cite{Adv.} establishes the optimal bounds for the S-spectrum of a selfadjoint quaternionic operator. Also, it proves that the S-specta of $ST$ and $TS$, along with zero, as well as their  S-spectral radii are the same.
\begin{Theorem}\label{Bounds}
\label{compact interval}\cite{Adv.}
 Let $\mathsf{H}$ be a right quaternionic Hilbert space. For $S, T\in\mathfrak{B}(\mathsf{H})$, we have
\begin{itemize}
\item[(i)]
If $T$ is a selfadjoint quaternionic operator and
\begin{equation*}
\label{infsup}
m_T=\inf \{\langle Tx,x\rangle\,:  \Vert x\Vert=1,\,x\in \mathsf{H}\}, \, \mbox{and}\,\, M_T=\sup \{\langle Tx,x\rangle\, : \Vert x\Vert=1,\,x\in \mathsf{H}\},
\end{equation*}
then $\sigma_S(T)\subset [m_T, M_T]$, $m_T=\min \sigma_S(T)$ and  $M_T=\max \sigma_S(T)$(see Theorem 3 in \cite{Adv.}).
\item[(ii)]
$\sigma_S(ST)\cup \{0 \}=\sigma_S(TS)\cup \{0 \}$ and $r_S(ST)=r_S(TS)$ (see Theorem 5  in \cite{Adv.}).
\end{itemize}
\end{Theorem}
%&&&&&&&&&&&&&&&&&&&&&&&&&&&&&&&&&&&&&&&&&&&&&&&&&&&&&&&&&&&&&&&&&&&&

\noindent The following lemma records two elementary facts about the order structure on $\mathfrak{B}(\mathsf{H})$ that we use repeatedly in the proof of Theorem \ref{Hansen} below, to make explicit that every operator to which $f$ is applied indeed has S-spectrum contained in $I$.

\begin{Lemma}\label{spec-order}
Let $S\in\mathfrak{B}(\mathsf{H})$ be selfadjoint and $m,M\in\mathbb{R}$ with $m\le M$. Then
\begin{itemize}
\item[(a)] $\sigma_S(S)\subset[m,M]$ if and only if $mI\le S\le MI$.
\item[(b)] For any right quaternionic Hilbert space $\mathsf{H}_1$ and any $C\in\mathfrak{B}(\mathsf{H}_1,\mathsf{H})$, if $S\ge 0$ then $C^*SC\ge 0$ on $\mathsf{H}_1$; if $S\le 0$ then $C^*SC\le 0$.
\end{itemize}
\end{Lemma}
\begin{proof}
(a) ($\Rightarrow$) If $\sigma_S(S)\subset[m,M]$ then $t\mapsto t-m$ and $t\mapsto M-t$ are nonnegative on $\sigma_S(S)$, so by positivity of the functional calculus (Theorem \ref{remark}(ii)), $S-mI\ge 0$ and $MI-S\ge 0$, i.e. $mI\le S\le MI$.\\
($\Leftarrow$) By Theorem \ref{Bounds} (i) is valid.\\
\noindent (b) For $x\in\mathsf{H}_1$, $\langle C^*SCx,x\rangle=\langle SCx,Cx\rangle$, which is $\ge0$ (resp. $\le0$) whenever $S\ge0$ (resp. $S\le0$), since this holds for every vector $Cx\in\mathsf{H}$.
\end{proof}

%%%%%%%%%%%%%%%%%%%%%%%%%%%%%%%%%%%%%%%%%%%%%%%%%%%%%%%&&&&&&&&&&&&&&&
%%%%%%%%%%%%%%%%%%%%%%%%%%%%%%%%%%%%%%%%%%%%%%%%%%%%%%
\noindent The following lemma is an extension of Lemma 1.6 in \cite{Mond} for the quaternionic setting.
\begin{Lemma} \label{U^*SU}
Let \(S\) be selfadjoint and let \(U\) be an
isometry, \(U^{*}U = I\). Suppose \(f\) is real-valued and
continuous on an interval containing \(\sigma_{S}(S) \cup \{ 0\}\). Then
\begin{equation}\label{U^*SU1}
f\left( USU^{*} \right) = Uf(S)U^{*} + f(0)\left( I - UU^{*} \right).\end{equation}
If \(U\) is unitary, then for every continuous real-valued \(f\) on the
relevant spectrum,
\begin{equation}\label{U^*SU2}f\left( U^{*}SU \right) = U^{*}f(S)U.\end{equation}
\end{Lemma}
\begin{proof}
Let \(P = UU^{*}\). For every integer \(k \geq 1\), by \(U^{*}U = I\), we have $\left( USU^{*} \right)^{k} = US^{k}U^{*}$. If $p(t) = a_{0} + \sum_{k = 1}^{n}a_{k}t^{k}$, 
then
\[p\left( USU^{*} \right) = a_{0}I + U\left( \sum_{k = 1}^{n}a_{k}S^{k} \right)U^{*}.\]
Therefore,
\[p\left( USU^{*} \right) = Up(S)U^{*} + a_{0}\left( I - UU^{*} \right).\]
Uniform approximation by real polynomials on the compact set
\(\sigma_{S}(S) \cup \{ 0\}\) gives
\[f\left( USU^{*} \right) = Uf(S)U^{*} + f(0)\left( I - UU^{*} \right).\]
\end{proof}
\noindent The following corollary establishes the polar decomposition of $T^*$ based on the polar decomposition of $T$.
\begin{Corollary}\label{T^*} \cite{JANA}
Let $T \in  \mathfrak{B}(\mathsf{H})$ and $T=U\vert T\vert$ be its polar decomposition. Then
\begin{itemize}
\item[(i)] $\vert T^* \vert^t= U\vert T\vert ^tU^*$ for all positive numbers $t$.
\item[(ii)] $T^*= U^* \vert T^*\vert$ is the polar decomposition of $T^*$.
\end{itemize}
\end{Corollary}
%%%%%%%%%%%%%%%%%%%%%%%%
\noindent We recall below the quaternionic operator inequalities established in the previous works, which form the foundational toolkit for quaternionic operator theory. The quaternionic L\"owner-Heinz and H\"older-McCarthy inequalities will be used directly in the proofs of our main results; the quaternionic Furuta and Kantorovich inequalities are included for completeness.

%%%%%%%%%%%%%%%%%%%%%%%%%%%%
\begin{Theorem}\label{H-M-L}
Let $S, T\in\mathfrak{B}(\mathsf{H})$ and $S\geq T\geq 0$, then for all $r \geq 0$,
\begin{itemize}
\item[(i)] \cite{Adv.} Quaternionic H\"{o}lder-McCarthy's inequality (H-M inequality): \\
For all unit vectors $x\in \mathsf{H}$
\begin{itemize}
\item[(a)]
$\langle T^r x, x\rangle \geq \langle Tx, x\rangle ^r $ for all $r>1$,
\item[(b)] $\langle T^r x, x\rangle \leq \langle Tx, x\rangle ^r$, for all $r\in (0,1).$
\end{itemize}
\item[(ii)] \cite{Adv.} Quaternionic  L\"{o}wner-Heinz inequality (L-H inequality): $S^r\geq T^r$ for all $r\in [0,1]$.
\item[(iii)] \cite{JANA} Quaternionic Furuta inequality (F inequality): If $p\geq 0 $, $ q\geq 1$ and $(1+2r)q \geq p+2r$ then
\[(T^rS^pT^r)^{\frac{1}{q}}\geq T^{\frac{p+2r}{q}} \qquad \mbox{ and }\qquad  S^{{\frac{p+2r}{q}}}\geq (S^rT^pS^r)^{\frac{1}{q}}.\]
\item[(iv)] \cite{Kont} Quaternionic Kantorovich inequality: If  $\sigma_S(T)\subset [m, M]$, for $0 <m<M$, then for all unit vectors $x\in \mathsf{H}$
\[ \langle Tx, x\rangle \langle T^{-1}x, x\rangle \leq \frac{(M+m)^2}{4Mm}.\]

\end{itemize}
\end{Theorem}
\noindent The results of the following theorem are quaternionic analogues of classical results derived from the Kantorovich inequality. These results are included for completeness and to illustrate the breadth of quaternionic operator inequalities; however, they are not required for the proofs of our main results in the subsequent sections.\\
\noindent
\noindent Part (i) in the following theorem is the quaternionic extension of Theorem 1.29 in \cite{Mond} and the second part is  the quaternionic extension of Theorem 8.1 in \cite{Mond}.
\begin{Theorem}\label{resultkont2}
Let $S, T\in\mathfrak{B}(\mathsf{H})$ and $S\geq T > 0$, and $\sigma_S(T)\subset [m,M]$ for $0<m \leq M$. Then
\begin{itemize}
\item[(i)] $\langle T^2x, x\rangle \leq \frac{(M+m)^2}{4Mm} \langle Tx, x\rangle^2$, for all unit vectors $x\in \mathsf{H}$;
\item[(ii)] $\frac{(M+m)^2}{4Mm} S^2 \ge T^2$.
\end{itemize}
\end{Theorem}
\begin{proof} (i) By substituting $\frac{T^{\frac{1}{2}}x}{\Vert T^{\frac{1}{2}}x\Vert }$ for a unit vector $x$ in the Kantorovich inequality the desired inequality follows. \\
\noindent (ii) For every unit vector $x\in \mathsf{H}$ we have
\begin{eqnarray*}
\langle T^2x, x\rangle &\leq & \frac{(M+m)^2}{4Mm} \langle Tx, x\rangle^2 \qquad  \mbox{(By (i))} \\
& \leq & \frac{(M+m)^2}{4Mm} \langle Sx, x\rangle ^2 \qquad \mbox{(Since}\, S\geq T > 0) \\
& \leq & \frac{(M+m)^2}{4Mm} \langle S^2x, x\rangle. \qquad \mbox{(by H-M inequality)}
\end{eqnarray*}
\end{proof}
%%%%%%%%%%%%%%%%%%%%%%%%%%%%%%%%%%%%%%%%%%%%%%%%%%%%%%%%%%%%%%%%
\section{Quaternionic Hansen--Pedersen Inequality}
Let \(\mathsf H_1, \mathsf H_2\) be right quaternionic Hilbert spaces with inner products \(\langle \cdot, \cdot \rangle_1\) and \(\langle \cdot, \cdot \rangle_2\), respectively. Their orthogonal direct sum is the right quaternionic Hilbert space
\[
\mathsf H := \mathsf H_1 \oplus \mathsf H_2 := \{ (x_1, x_2) : x_1 \in \mathsf H_1, \; x_2 \in \mathsf H_2 \},
\]
equipped with componentwise addition, right quaternionic scalar multiplication
\[
(x_1, x_2) q := (x_1 q, x_2 q), \qquad q \in \mathbb H,
\]
and inner product
\[
\langle (x_1, x_2), (y_1, y_2) \rangle_{\mathsf H}
:= \langle x_1, y_1 \rangle_1 + \langle x_2, y_2 \rangle_2.
\]
The norm is given by \(\| (x_1, x_2) \|_{\mathsf H}^2 = \|x_1\|_1^2 + \|x_2\|_2^2\).
For \(T_1 \in \mathfrak B(\mathsf H_1)\) and \(T_2 \in \mathfrak B(\mathsf H_2)\), their direct sum is the operator
\[
T_1 \oplus T_2 \in \mathfrak B(\mathsf H_1 \oplus \mathsf H_2), \qquad (T_1 \oplus T_2)(x_1, x_2) := (T_1 x_1, T_2 x_2).
\]
The operator norm satisfies
\[
\|T_1 \oplus T_2\|_{\mathfrak B(\mathsf H_1 \oplus \mathsf H_2)}
= \max\big\{ \|T_1 \|_{\mathfrak B(\mathsf H_1)},\; \|T_2\|_{\mathfrak B(\mathsf H_2)} \big\}.
\]
Moreover, if \(T_1\) and \(T_2\) are self-adjoint, then \(T_1 \oplus T_2\) is self-adjoint on \(\mathsf H_1 \oplus \mathsf H_2\). The direct sum is invertible if and only if both \(T_1\) and \(T_2\) are invertible, and in that case
$(T_1 \oplus T_2)^{-1} = T_1^{-1} \oplus T_2^{-1}$.
\begin{Proposition}\label{directsum} (Quaternionic Functional Calculus on Direct Sums)
Let $\mathsf{H}_1$ and $\mathsf{H}_2$ be right quaternionic Hilbert spaces, $T_1\in\mathfrak{B}(\mathsf{H}_1)$, $T_2\in\mathfrak{B}(\mathsf{H}_2)$ be self-adjoint, $T:=T_1\oplus T_2\in\mathfrak B(\mathsf H_1\oplus\mathsf H_2)$. Then for every $f\in \mathcal{C}(\sigma_S(T_1)\cup\sigma_S(T_2);\mathbb{R})$,
\begin{equation}
f(T_1\oplus T_2)=f(T_1)\oplus f(T_2).
\end{equation}
\end{Proposition}
\begin{proof}
Let \(p(x) = \sum_{k=0}^n a_k x^k\) be any real polynomial (i.e., \(a_k \in \mathbb R\) for all \(k\)). For any \(m \ge 0\), the direct sum acts componentwise, so
$(T_1 \oplus T_2)^m = T_1^m \oplus T_2^m$.
Therefore,
\[\begin{aligned}
p(T_1 \oplus T_2)
&= \sum_{k=0}^n a_k (T_1 \oplus T_2)^k = \sum_{k=0}^n a_k (T_1^k \oplus T_2^k)= \left( \sum_{k=0}^n a_k T_1^k \right) \oplus \left( \sum_{k=0}^n a_k T_2^k \right)\\
&= p(T_1) \oplus p(T_2).
\end{aligned}
\]
Thus the identity holds for every real polynomial
\begin{equation}\label{p}
p(T_1 \oplus T_2) = p(T_1) \oplus p(T_2).
\end{equation}
We next claim that
\begin{equation}\label{S}
\sigma_S(T_1 \oplus T_2) = \sigma_S(T_1) \cup \sigma_S(T_2).
\end{equation}
\noindent By Definition \ref{spectrum} and direct computation $\Delta_q(T_1 \oplus T_2) = \Delta_q(T_1) \oplus \Delta_q(T_2)$.
Since a direct sum of operators is invertible if and only if each summand is invertible, we have
\[
\begin{aligned}
q \notin \sigma_S(T_1 \oplus T_2)
&\iff \Delta_q(T_1 \oplus T_2) \text{ is invertible} \\
&\iff \Delta_q(T_1) \oplus \Delta_q(T_2) \text{ is invertible} \\
&\iff \Delta_q(T_1) \text{ is invertible and } \Delta_q(T_2) \text{ is invertible} \\
&\iff q \notin \sigma_S(T_1) \text{ and } q \notin \sigma_S(T_2) \\
&\iff q \notin \sigma_S(T_1) \cup \sigma_S(T_2).
\end{aligned}
\]
Hence (\ref{S}) holds.
\noindent Since \(T_1\) and \(T_2\) are self-adjoint, their S-spectra are compact subsets of \(\mathbb R\) (see, e.g., \cite{Col-book}, Theorem 4.3). Therefore,
$
K := \sigma_S(T_1) \cup \sigma_S(T_2) \subset \mathbb R
$
is compact.\\
Finally, if \(f \in C(K; \mathbb R)\), then by  the classical Weierstrass approximation theorem, there exists a sequence of real polynomials \(\{p_n\}_{n=1}^\infty\) converging uniformly to $f$ on $K$, which means
\[
\lim_{n \to \infty} \max_{x \in K} |p_n(x) - f(x)| = 0.
\]
By continuity of the functional
calculus in Theorem 2, we imply
\[
p_n(A) \to f(A) \quad \text{in operator norm}.
\]
Applying this to \(A = T_1\), \(A = T_2\), and \(A = T_1 \oplus T_2\), and using (\ref{S}), we obtain
\[
p_n(T_1) \to f(T_1), \qquad
p_n(T_2) \to f(T_2), \qquad
p_n(T_1 \oplus T_2) \to f(T_1 \oplus T_2)
\]
in operator norm. From (\ref{p}), we have $p_n(T_1 \oplus T_2) = p_n(T_1) \oplus p_n(T_2)$, for each \(n\). Using the fact that the direct sum is continuous with respect to the operator norm, we get
\[
f(T_1 \oplus T_2) = f(T_1) \oplus f(T_2).
\]
\end{proof}
\noindent A continuous real-valued function \(f\) on an interval \(I\) is
quaternionic operator convex if, for every right quaternionic Hilbert
space \(\mathsf{K}\), every selfadjoint
\(A,B \in B\left( \mathsf{K} \right)\) with S-spectra contained in
\(I\), and every \(\lambda \in \lbrack 0,1\rbrack\),
\[f\left( (1 - \lambda)A + \lambda B \right) \leq (1 - \lambda)f(A) + \lambda f(B).\]
It is operator concave if the inequality is reversed.
\begin{Theorem}\label{Hansen} (Quaternionic Hansen--Pedersen Inequality) Let
\(\mathsf{H}\) be a right quaternionic Hilbert space. Let
\(T \in B\left( \mathsf{H} \right)\) be selfadjoint with $\sigma_{S}(T) \subset I$, where \(I \subset \mathbb{R}\) is an interval, and let
\(f:I \rightarrow \mathbb{R}\) be quaternionic operator convex. If
\(U \in B\left( \mathsf{H} \right)\) is an isometry,
\(U^{*}U = I\), then
\begin{equation}\label{HI}
 f\left( U^{*}TU \right) \leq U^{*}f(T)U.
\end{equation}
\end{Theorem}
\begin{proof} Define operators $A$ and $B$ on $\mathsf{H}\oplus\mathsf{H}$ by
\begin{equation*}
A=\left[\begin{array}{cc}
U&V\\
0&-U^*
\end{array}\right],\,\,\,
B=\left[\begin{array}{cc}
U&-V\\
0&U^*
\end{array}\right],
\end{equation*}
where $V=(I_{\mathsf{H}}-UU^*)^{1/2}$. Therefore, $V^2+UU^*=I_{\mathsf{H}}$, $U^*V^2=U^*-U^*UU^*=0$ and $V^2U=U-UU^*U=0$. Since $V$ is selfadjoint, for every $u \in \mathsf{H}$ we have
\[
\langle V^2u, u \rangle = \langle Vu, V^*u \rangle = \langle Vu, Vu \rangle = \Vert Vu\Vert ^2.
\]
Applying this with $u = Uu_0$ for arbitrary $u_0 \in \mathsf{H}$, and using $V^2U = 0$, gives $\Vert V(Uu_0) \Vert^2 = \langle V^2 Uu_0, Uu_0 \rangle = 0$, so $VUu_0 = 0$. As $u_0$ was arbitrary, $VU = 0$. Taking adjoints and using $V^* = V$ gives $U^*V = (VU)^* = 0$. It is easy to check that $A^*=\left[\begin{array}{cc}
U^*&0\\
V&-U
\end{array}\right]$. We claim that $A$ and $B$ are unitary operators on $\mathsf{H}\oplus\mathsf{H}$. To see this,
\begin{equation*}
AA^*=\left[\begin{array}{cc}
U&V\\
0&-U^*
\end{array}\right]
\left[\begin{array}{cc}
U^*&0\\
V&-U
\end{array}\right]=
\left[\begin{array}{cc}
UU^*+V^2&-VU\\
-U^*V&U^*U
\end{array}\right]=
\left[\begin{array}{cc}
I_{\mathsf{H}}&0\\
0&I_{\mathsf{H}}
\end{array}\right]=I_{\mathsf{H}\oplus\mathsf{H}}.
\end{equation*}

\begin{align*}
A^*A &=
\left[\begin{array}{cc}
U^* & 0 \\
V & -U
\end{array}\right]
\left[\begin{array}{cc}
U & V \\
0 & -U^*
\end{array}\right] =
\left[\begin{array}{cc}
U^*U & U^*V \\
VU & V^2 + UU^*
\end{array}\right] =
\left[\begin{array}{cc}
I_H & 0 \\
0 & I_H
\end{array}\right]= I_{H \oplus H}.
\end{align*}
The identical computation applied to $B$ gives $B^*B =I_{\mathsf{H}\oplus\mathsf{H}}= BB^*$. so $B$ is unitary too.
Let $X=\left[\begin{array}{cc}
T&0\\
0&T
\end{array}\right]$. Then
\begin{equation*}
A^*XA=\left[\begin{array}{cc}
U^*TU&U^*TV\\
VTU&VTV+UTU^*
\end{array}\right]\, \,\,\, \mbox{and}\,\,\,\,
B^*XB=\left[\begin{array}{cc}
U^*TU&-U^*TV\\
-VTU&VTV+UTU^*
\end{array}\right].
\end{equation*}
So,
\begin{equation*}
\frac{A^*XA+B^*XB}{2}=\left[\begin{array}{cc}
U^*TU&0\\
0&VTV+UTU^*
\end{array}\right].
\end{equation*}
Before applying $f$, we check that every operator involved has S-spectrum in $I$.
Let $m = \min\sigma_{S}(T)$ and $M = \max\sigma_{S}(T)$, then $\lbrack m,M\rbrack \subset I$, by Theorem \ref{Bounds} (i), and by Lemma \ref{spec-order}, $mI \leq T \leq MI$. Since $U^{*}U = I$, we get $mI \leq U^{*}TU \leq MI$. Next put $D = VTV + UTU^{*}$. Since $V^{2} + UU^{*} = I$,
we obtain
\[D - mI = V(T - mI)V + U(T - mI)U^{*} \geq 0,\]
and similarly
\[MI - D = V(MI - T)V + U(MI - T)U^{*} \geq 0.\]
Therefore $mI \leq D \leq MI$. Applying Lemma \ref{spec-order} again gives $\sigma_{S}\left( U^{*}TU \right) \subset \lbrack m,M\rbrack \subset I$, and $\sigma_{S}(D) \subset \lbrack m,M\rbrack \subset I$. For the unitary conjugates, use direct unitary invariance. If
\(Y = A^{*}XA\), then for every \(q \in \mathbb{H}\),
\[\Delta_{q}(Y) = A^{*}\Delta_{q}(X)A.\]
Hence \(\Delta_{q}(Y)\) is invertible if and only if \(\Delta_{q}(X)\)
is invertible, and consequently
\[\sigma_{S}\left( A^{*}XA \right) = \sigma_{S}(X).\]
The same argument applies to \(B\). Thus $f$ is legitimately applied to each operator below via Theorem \ref{remark}. By operator convexity of \(f\), Proposition
  \ref{directsum}, and unitary invariance of the functional calculus,
\begin{eqnarray*}
  \left[\begin{array}{cc}
f(U^*TU)&0\\
0&f(VTV+UWU^*)
\end{array}\right]&=&f\left[\begin{array}{cc}
U^*TU&0\\
0&VTV+UWU^*
\end{array}\right]\\
&=& f\left (\frac{A^*XA+B^*XB}{2}\right )\\
&\leq & \frac{f(A^*XA)+f(B^*XB)}{2}= \frac{A^*f(X)A+B^*f(X)B}{2}\\
&=& \left[\begin{array}{cc}
U^*f(T)U&0\\
0&V^*f(T)V+Uf(T)U^*
\end{array}\right].
\end{eqnarray*}
Comparison of the
  \((1,1)\) entries of the latter inequality shows that $f(U^*TU)\leq U^*f(T)U$.
\end{proof}
\begin{Remark}\label{Concave}

Since \(f\) is operator concave if and only if \(- f\) is operator
convex, Theorem 10 yields the reversed inequality.
It is also called Hansen--Pedersen inequality.
\end{Remark}
\section{An Application of Quaternionic Hansen--Pedersen Inequality}
\noindent We now apply the quaternionic Hansen--Pedersen inequality to the
transform \(S_{r}(T) = U|T|^{r}U\), that was defined in the complex case by Furuta in \cite{FSRT}. The following theorem is the extension of Theorem 3.1 in \cite{Srt} to the quaternionic setting.
\begin{Theorem}\label{Menkad}
Let $T\in\mathfrak{B}(\mathsf{H})$ be an injective $p$-hyponormal operator with $p \in (0,1]$ and $T= U\vert T\vert$ be its polar decomposition. Then for $r\in (0, \frac{1}{2}]$ the following assertions hold:
\begin{itemize}
\item[(i)] If $0 < p \le \frac{1}{2}$, then $S_r(T)= U |T|^r U$ is $2p$-hyponormal.
\item[(ii)] If $\frac{1}{2} < p \le 1$, then $S_r(T)$ is hyponormal.
\end{itemize}
\end{Theorem}
\begin{proof}
(i)  Injectivity of $T$ guarantees that $U$ is an isometry, i.e. $U^*U=I$ on $\mathsf{H}$. Assume $0 < p \le \frac{1}{2}$. Since $T$ is $p$-hyponormal, we have $|T|^{2p} \ge |T^*|^{2p}$. By Corollary \ref{T^*}, this is equivalent to $|T|^{2p} \ge U |T|^{2p} U^*$. Since $0<2r\leq 1$, by L-H inequality $(|T|^{2p})^{2r} \ge (U |T|^{2p} U^*)^{2r}$. By equality \eqref{U^*SU1} in Lemma \ref{U^*SU}  the right-hand side becomes $(U |T|^{2p} U^*)^{2r} = U (|T|^{2p})^{2r} U^*$.
Hence, $|T|^{4pr} \ge U |T|^{4pr} U^*$. Multiplying the latter inequality on the left by $U^*$ and on the right by $U$, we get $U^* |T|^{4pr} U \ge |T|^{4pr}$. Therefore
\begin{equation}\label{a}
U^* |T|^{4pr} U \ge |T|^{4pr} \ge U |T|^{4pr} U^*.
\end{equation}
\noindent On the other hand we have
\begin{eqnarray*}
(S_r(T)^* S_r(T))^{2p} & = & (U^* |T|^r U^* U |T|^r U )^{2p} = (U^* |T|^{2r} U)^{2p}\\
& \geq & U^* (|T|^{2r})^{2p} U = U^* |T|^{4pr} U \qquad (\mbox{By Remark \ref{Concave} with} f(t) = t^{2p}) \\
& \geq & U |T|^{4pr} U^* \qquad (\mbox{By \eqref{a}}) \\
& = & (U |T|^{2r} U^*)^{2p}. \qquad (\mbox{By \eqref{U^*SU1} in Lemma  \ref{U^*SU}})
\end{eqnarray*}
In the next step, since $I-UU^*$ is a positive operator, $U |T|^r U U^* |T|^r U^*\le U|T|^r I|T|^rU^*=U|T|^{2r}U^*$; applying the L-H inequality (valid since $2p\le1$) to raise both sides to the power $2p$ gives
\begin{eqnarray*}
(S_r(T) S_r(T)^*)^{2p} &=& (U |T|^r U U^* |T|^r U^* )^{2p} \\
&\leq &(U |T|^{2r} U^*)^{2p}.
\end{eqnarray*}
\noindent Hence,
\[
(S_r(T)^* S_r(T))^{2p} \ge (S_r(T) S_r(T)^*)^{2p}.
\]
Thus $S_r(T)$ is $2p$-hyponormal.\\
(ii) Assume $\frac{1}{2} < p \le 1$. By the monotonicity of $p$-hyponormality (see Proposition 9 in \cite{JANA}), $T$ is $\frac{1}{2}$-hyponormal. Applying Part (i) with $p = \frac{1}{2}$, we obtain that $S_r(T)$ is $1$-hyponormal, i.e., hyponormal. This completes the proof.
\end{proof}
\begin{Example}
Let $\ell^2(\mathbb{H})$ denote the right quaternionic Hilbert space of square-summable sequences with orthonormal basis $\{e_n\}_{n \ge 0}$. For a sequence of positive reals $\{w_n\}_{n \ge 0}$, define the weighted shift $T_w \in B(\ell^2(\mathbb{H}))$ by
\[
T_w e_n = e_{n+1} w_n.
\]
Since the weights are real (hence central), a direct computation as in the complex case gives
\[
T_w^* e_n = e_{n-1} w_{n-1} \quad (n \ge 1), \qquad T_w^* e_0 = 0,
\]
so
\[
T_w^* T_w e_n = w_n^2 e_n, \qquad T_w T_w^* e_n = w_{n-1}^2 e_n \quad (w_{-1} := 0).
\]
Both operators are diagonal in $\{e_n\}_{n \ge 0}$, hence $(T_w^* T_w)^p$ and $(T_w T_w^*)^p$ are simultaneously diagonal with eigenvalues $w_n^{2p}$ and $w_{n-1}^{2p}$. Thus $T_w$ is $p$-hyponormal iff $\{w_n\}$ is increasing, for any $p > 0$.
A simple bounded choice is $w_{n} = 2 - \frac{1}{n + 1},\quad\quad n \geq 0$. Then
\[1 = w_{0} < w_{1} < w_{2} < \cdots < 2,\]
so $\{w_n\}_{n \ge 0}$ is positive, bounded and increasing. Hence
the corresponding weighted shift \(T_{w}\) is bounded, injective and
\(p\)-hyponormal for every \(p > 0\). It is nonnormal because
\[T_{w}^{*}T_{w}e_{0} = w_{0}^{2}e_{0},\quad\quad T_{w}T_{w}^{*}e_{0} = 0.\]
For the polar decomposition, $Ue_{n} = e_{n + 1}$ and $\left| T_{w} \right|e_{n} = w_{n}e_{n}$. Therefore
\[S_{r}\left( T_{w} \right)e_{n} = w_{n + 1}^{r}e_{n + 2}.\]
Consequently,
\[S_{r}\left( T_{w} \right)^{*}S_{r}\left( T_{w} \right)e_{n} = w_{n + 1}^{2r}e_{n},\]
while
\[S_{r}\left( T_{w} \right)S_{r}\left( T_{w} \right)^{*}e_{n} = \left\{ \begin{matrix}
0, & n = 0,1, \\
w_{n - 1}^{2r}e_{n}, & n \geq 2.
\end{matrix} \right.\ \]
\noindent Since $\{w_n\}_{n \ge 0}$ is increasing, this example actually
gives a hyponormal \(S_{r}\left( T_{w} \right)\) for every \(r > 0\). 
\end{Example}

\noindent In the following theorem we generalize Theorem 2.3 in \cite{Patel} to the quaternionic transform $S_r(T)$.

\begin{Theorem}
Let $T\in\mathcal{LH}$  with polar decomposition $T = U |T|$, then  $S_r(T)\in\mathcal{LH}$ for $r>0$.
\end{Theorem}

\begin{proof}
Log-hyponormality assumes \(T\) invertible, hence  \(U\) is unitary and \(|T|\) is positive invertible. Therefore \(S_{r}(T)\) is also invertible, so both logarithms in the conclusion are well-defined.
By \eqref{U^*SU1} in Lemma \ref{U^*SU} we have
\begin{eqnarray*}
|S_r(T)|^2 = S_r(T)^* S_r(T)= (U^* |T|^r U^*)(U |T|^r U)= U^* |T|^{2r} U= (U^* |T|^{r} U)^2.
\end{eqnarray*}
Since \(U^{*}|T|^{r}U\) is positive,  uniqueness of the  positive square root yields $\left| S_{r}(T) \right| = U^{*}|T|^{r}U$. Similarly, $|S_r(T)^*| = U |T|^r U^*$. Taking logarithms and using equality \eqref{U^*SU2} in Lemma \ref{U^*SU} for $f(t)=\log t$, we obtain
\begin{equation*}
\log |S_r(T)| = \log(U^* |T|^r U)= U^* (\log |T|^r) U= r \, U^* (\log |T|) U,
\end{equation*}
and
\begin{equation*}
\log |S_r(T)^*| = \log(U |T|^r U^*)= r \, U (\log |T|) U^*.
\end{equation*}
Since $T $ is log-hyponormal, we have $\log |T| \ge \log |T^*|$. But $ |T^*| = U |T| U^* $, by Corollary \ref{T^*} (i), so $\log |T^*| = U (\log |T|) U^*$. Thus
$
\log |T| \ge U (\log |T|) U^*
$.
Multiplying on the left by \( U^* \) and on the right by \( U \) gives $U^* (\log |T|) U \ge \log |T|$. Combining the latter two inequalities yields
\begin{equation*}
U^* (\log |T|) U \ge \log |T| \ge U (\log |T|) U^*.
\end{equation*}
Therefore,
\[
U^* (\log |T|) U \ge U (\log |T|) U^*.
\]
Since \( r > 0 \), multiplying by \( r \) preserves the inequality, so
\[
r \, U^* (\log |T|) U \ge r \, U (\log |T|) U^*.
\]
That is,
\[
\log |S_r(T)| \ge \log |S_r(T)^*|,
\]
which proves that \( S_r(T) \) is log-hyponormal.
\end{proof}
%%%%%%%%%%%%%%%%%%%%%%%%%%%%%%%%%%%%%%%%%%
\section{Conclusion}
We established a quaternionic Hansen--Pedersen inequality for real
operator-convex functions on right quaternionic Hilbert spaces and
obtained its operator-concave counterpart. The proof combines the
continuous real functional calculus for selfadjoint quaternionic
operators with a direct-sum identity and a block-unitary argument. As an
application, we showed that if \(T\) is an injective \(p\)-hyponormal
operator and \(0 < r \leq 1/2\), then \(S_{r}(T) = U|T|^{r}U\) is
\(2p\)-hyponormal for \(0 < p \leq 1/2\) and hyponormal for
\(1/2 < p \leq 1\). We also proved that the transform preserves
log-hyponormality for \(r > 0\). These results suggest that quaternionic
Jensen-type inequalities can provide a useful framework for studying
further transforms associated with polar decomposition.

%%%%%%%%%%%%%%%%%%%%%%%%%%%%%%%%%%%%%%%%%%%%%%%%%%%%%%%%%%%
%\section*{Acknowledgements}

%%%%%%%%%%%%%%%%%%%%%%%%%%%%%%%%%%%%%%%%%%%%%%%%%%

\bibliographystyle{amsplain}

\end{document}